\documentclass[twoside,a4paper,12 pt]{amsart}

\usepackage{amssymb,amsmath}
\usepackage{graphicx}
\usepackage[lite]{amsrefs}
\usepackage[all]{xy}
\usepackage{geometry}
\usepackage{tikz}
\usepackage{color}
\usetikzlibrary{matrix,arrows}

\newcommand{\ringO}{\mathcal{O}}
\newcommand{\Hy}{\mathcal{H}}
\newcommand{\C}{{\mathbf{C}}}
\newcommand{\R}{{\mathbf{R}}}
\newcommand{\Z}{{\mathbf{Z}}}
\newcommand{\N}{{\mathbf{N}}}

\newcommand{\T}{{\bf{T}}}
\newcommand{\rationals}{{\mathbf{Q}}}

\theoremstyle{plain}
\newtheorem{theorem}{Theorem}
\newtheorem{lemma}[theorem]{Lemma}
\newtheorem{question}[theorem]{\bfseries Question}

\newtheorem{corollary}[theorem]{Corollary}

\theoremstyle{definition}

\theoremstyle{remark}

\begin{document}

\title{Bianchi's additional symmetries}
\author{Alexander D. Rahm}
\address{Universit\'e de la Polyn\'esie Fran\c{c}aise, Laboratoire GAATI}
\urladdr{http://gaati.org/rahm/}
\email{Alexander.Rahm@upf.pf}

\date{\today}

\begin{abstract}
In a 2012 note in Comptes Rendus Math\'ematique, the author did try to answer a question of \mbox{Jean-Pierre Serre};
it has recently been announced that the scope of that answer needs an adjustment, 
and the details of this adjustment are given in the present paper.
The original question is the following.
Consider the ring of integers~$\ringO$ in an imaginary quadratic number field,
and the Borel--Serre compactification of the quotient of hyperbolic 
$3$--space by $\mathrm{SL_2}(\ringO)$.
Consider the map~$\alpha$
 induced on homology when attaching the boundary
 into the Borel--Serre compactification.\\
\mbox{\emph{How can one determine the kernel of~$\alpha$ (in degree 1) ?}}\\
Serre used a global topological argument and obtained the rank of the kernel of~$\alpha$.
He added the question what submodule precisely this kernel is.
\end{abstract}

\maketitle

\section*{Introduction}
As the note ``On a question of Serre''~\cite{questionOfSerre} by the author had been published prematurely,
the proof of the key lemma remaining sketchy,
the author did contact an expert for the Borel--Serre compactification, Lizhen Ji,
for the project of establishing a detailed version of that key proof.
Lizhen Ji then found out that there are cases where that proof does not apply.
So it is necessary to correct the scope of the lemma, which has been announced in Comptes Rendus Math\'ematique~\cite{corrigendum},
and is detailed in the present paper.

The Bianchi groups, $\mathrm{SL_2}(\ringO)$ over the ring $\ringO$ of integers in an imaginary quadratic field $\rationals(\sqrt{-D})$, 
where $D$ is a square-free positive natural integer,
act naturally on hyperbolic $3$-space $\Hy$ (cf. the monographs~\cite{ElstrodtGrunewaldMennicke},~\cite{Fine},~\cite{MaclachlanReid}).
For a subgroup $\Gamma$ of finite index in $\mathrm{SL_2}(\ringO)$,
 consider the Borel--Serre compactification $_\Gamma \backslash \widehat{\Hy}$ of the orbit space $_\Gamma \backslash \Hy$,
constructed by Borel and Serre~\cite{BorelSerre}, which is a manifold with boundary.
In our case, the boundary components of $_\Gamma \backslash \widehat{\Hy}$ are disjoint,
which allows for an explicit description (see the appendix of Serre's paper~\cite{Serre}).
\textit{Throughout this paper, we will exclude the ring $\ringO$ from being the Gaussian integers, in $\rationals(\sqrt{-1})$,
or the Eisensteinian integers, in $\rationals(\sqrt{-3})$.}
In those two special cases, the boundary at the single cusp is a 2--sphere, and these two special cases can easily be treated separately by an explicit manual computation. 
So we are in the setting where each boundary component of $_\Gamma \backslash \widehat{\Hy}$ is a 2-torus $\T_\sigma$ which compactifies the cusp $\sigma$.
\begin{question}[Serre \cite{Serre}]\label{question}
 Consider the map~$\alpha$
 induced on homology when attaching the boundary
 into the Borel--Serre compactification $_\Gamma \backslash \widehat{\Hy}$.
How can one determine the kernel of $\alpha$ (in degree 1) ?
\end{question}

\textit{Motivation of Question~\ref{question}}.
The search for a set of generators for the kernel of the attaching map has motivations that are indicated already in Serre's 1970 paper.
The paper does (after achieving its goal, namely solving the congruence subgroups problem), 
describe how to calculate the Abelianisation of the investigated arithmetic groups in terms of generators for the kernel of the attaching map. 
Knowledge about the Abelianisation $\Gamma^{\rm ab}$ of $\Gamma$ a (congruence subgroup in a) Bianchi group, especially in terms of generating matrices, has the following application to the Langlands programme 
(we can assume that this application was already realised by Serre at that time, because of his subsequent question in the paper).
For each weight $k$ of modular forms, there is a coefficient module $M_k$ such that the Eichler--Shimura--Harder isomorphism identifies the space of weight $k$ modular forms for $\Gamma$
with the cohomology $H^1(\Gamma; M_k)$. The latter can be computed explicitly as Hom$(\Gamma^{\rm ab} , M_k)$ when $\Gamma^{\rm ab}$ is given in terms of generating matrices.
These computations would then not have to pass by an explicit computation of a fundamental domain like the Bianchi fundamental polyhedron,
which is  a necessary step still nowadays~\cites{RahmSengun, RahmTsaknias} for computing the dimension of $H^1(\Gamma; M_k)$.

\medskip

The following lemma was the key for the approach to Question~\ref{question} pursued in the 2012 note. 
In the remainder of this paper, we consider an imaginary quadratic field $\rationals(\sqrt{-D})$ 
with $D$ a square-free positive natural integer, and we set $\Gamma = \mathrm{SL_2}(\ringO_{\rationals(\sqrt{-D})})$.
We decompose the $2$-torus $\T_\sigma$ in the classical way into a $2$--cell, two edges and a vertex.

\begin{lemma} \label{torus-lemma}
Let $n$ be the number of prime divisors of $D$. Let $N = 2^{n-1}$ for $D \equiv 3 \bmod 4$, $N = 2^{n}$ for $D \equiv 1\text{ or } 2 \bmod 4$. 
Then $_\Gamma \backslash \Hy$ admits at least $N$ cusps $\sigma$ such that the inclusion of~$\T_\sigma$ into $_\Gamma \backslash \widehat{\Hy}$ 
induces on~$H_1(\T_\sigma)$ a map of rank $1$.
\end{lemma}

The restriction to $N$ cusps was missing in the 2012 note, and we will give a completed proof in Section~\ref{The proof}, making use of this restriction.
For this purpose, we exploit symmetries of the quotient space $_\Gamma \backslash \Hy$, 
which Bianchi did find additionally to the ``basic'' ones which occur for all studied rings $\ringO$ 
and which are given by complex conjugation and a rotation of order $2$ around the origin of the complex plane (see Section~\ref{Bianchi's additional symmetries}). 

\begin{corollary} \label{exponent 2 case}
 If the class group of $\ringO$ is isomorphic to $\Z/2^m\Z$ for some $m \in \N$, then for all cusps $\sigma$ of $_\Gamma \backslash \Hy$,
 the inclusion of~$\T_\sigma$ into $_\Gamma \backslash \widehat{\Hy}$ 
induces on~$H_1(\T_\sigma)$ a map of rank $1$.
\end{corollary}
\begin{proof}
Let the numbers $n$ and $N$ be as in Lemma~\ref{torus-lemma}.
By~\cite{Cox}*{thm. 3.22}, our hypothesis on the class group of $\ringO$ is equivalent to $m = $\scriptsize $\begin{cases}
        n-1, & D \equiv 3 \bmod 4,\\
        n,& D \equiv 1 \text{ or } 2 \bmod 4.\\
       \end{cases}$\normalsize

 It is well known that each ideal class of $\ringO$ corresponds to one cusp of $_\Gamma \backslash \Hy$, so there are exactly $N$ cusps, 
 and Lemma~\ref{torus-lemma} yields the claim.
 \end{proof}
As the class group type assumed in Corollary~\ref{exponent 2 case} occurs only finitely many times~\cite{Weinberger},
it might seem disappointing that the scope of the theorem deduced from it in~\cite{questionOfSerre}
is much more narrow than it was originally claimed,
but this actually means that once that we leave this scope of validity, the topology of the Bianchi orbifolds near their boundary is much richer than asserted in~\cite{questionOfSerre},
and can provide interesting studies for future generations of mathematicians.

\section{Bianchi's additional symmetries} \label{Bianchi's additional symmetries}
In the cases $\Gamma = \mathrm{SL_2}(\ringO)$ which he considered, 
Luigi Bianchi did make an exhaustive description of the symmetries of the quotient space $_\Gamma \backslash \Hy$.
They are given by outer automorphisms of $\Gamma$, and a subgroup of these automorphisms is, for all studied rings $\ringO$, 
the Klein four-group with generators $c$ and $e$,
where for $\gamma \in \Gamma$, the matrix $c(\gamma) = \overline{\gamma}$ is the complex conjugate of $\gamma$, and $e(\gamma) = E\cdot \gamma \cdot E$ with
$E = \begin{pmatrix}    -1 & 0 \\ 
			 0 & 1 \end{pmatrix} 
			 \in \text{GL}_2(\ringO)$
-- see~\cite{Swan}*{remark below lemma 4.19}.
We will consider this Klein four-group as the ``basic'' symmetries, and will call ``additional'' the symmetries of $_\Gamma \backslash \Hy$
given by outer automorphisms outside this Klein four-group.
Throughout this paper, we use the upper-half space model $\Hy$ for hyperbolic $3$-space, as it is the one used by Bianchi.
As a set, $ \Hy = \{ (z,\zeta) \in \C \times \R \medspace | \medspace \zeta > 0 \}.$
Then, the basic symmetries take the form $c(z,\zeta) = (\overline{z},\zeta)$ and $e(z,\zeta) = (-z,\zeta)$. 

The list of Bianchi's symmetries presented in this section will allow us to reduce the proof in the following section 
in each case to the situation at the cusp at infinity.
As a fundamental domain in hyperbolic space, we make use of the polyhedron with missing vertices at the cusps,
described by Bianchi~\cite{Bianchi1892}, and which we will call the \emph{Bianchi fundamental polyhedron}.
It is the intersection of a fundamental domain in $\Hy$ for the stabiliser in $\Gamma$ of the cusp at $\infty$ (in the shape of a rectangular prism) 
with the set of points of $\Hy$ that are closer to~$\infty$ than to any other cusp, with respect to Mendoza's distance to cusps~\cite{Mendoza}.
We consider the cellular structure on $_\Gamma \backslash \widehat{\Hy}$
induced by the Bianchi fundamental polyhedron.
From Bianchi's articles~\cites{Bianchi1892, Bianchi1893}, 
 it was not clear that one can always compute the Bianchi fundamental polyhedron, 
 a fact which meanwhile has been established by an algorithm of Richard G. Swan~\cite{Swan} 
 that has been implemented at least twice~(\cite{BianchiGP}, which is a straight implementation of Swan's algorithm,
 and \cites{Aranes, CremonaAranes}, which additionally uses ideas of Cremona~\cite{Cremona}).
The reason why in some cases Bianchi did not obtain the Bianchi fundamental polyhedron, 
must be that some ``sfere di riflessione impropria'' did escape him
\cite{Bianchi1892}*{\S 20; at the first example, $D = 14$, it is the hemisphere of center $\frac{1}{2}+\frac{\sqrt{-14}}{3}$
and radius $\frac{1}{6}$}.
Note that in the upper-half space model, totally geodesic planes are modelled by vertical planes, respectively by hemispheres centered in the boundary, so we will call the latter ``reflecting (hemi)spheres''
to continue Bianchi's terminology, even though they are to be thought as mirroring planes in hyperbolic space.
An improper reflection sphere (``sfera di riflessione impropria'') has its reflection composed with the order-$2$-rotation around its vertical radius (the one connecting its center with its highest point).
Instead of the missing reflecting spheres, Bianchi did find additional symmetries which did allow him to construct a differently shaped but equivalent fundamental domain~\cite{Bianchi1893}.
These symmetries are going to be the foundations of our proof.
Since Bianchi's papers, a cusp is called \textit{singular} 
when it is not in the $\Gamma$-orbit of~$\infty$. 
\begin{theorem}[Bianchi] \label{Bianchi statement}
 Let $D$, $n$ and $N$ be as in Lemma~\ref{torus-lemma}.
Then there are $N-1$ symmetries of the Bianchi fundamental polyhedron, each flipping the cusp~$\infty$ with a different singular cusp,
while leaving invariant a hemisphere of radius smaller than~$1$,
and preserving the $\Gamma$-cell structure.
\end{theorem}

We will call those symmetries \textit{Bianchi's additional symmetries}. 

\begin{proof}[We sketch Bianchi's proof of Theorem~\ref{Bianchi statement} in order to fix notation.]
With $D$ a square-free positive natural integer satisfying either
\begin{itemize}
 \item Case A: $D \equiv 1$ or $2 \bmod 4$, 
	and suppose that $D$ admits a prime divisor $m < D$; 
 \item or Case B:  $D \equiv 1 \bmod 4$ with $D > 1$,
\end{itemize}
Bianchi specified the symmetry by composing complex conjugation with the action of the matrix 
\begin{center}
\hfill
 $M_\sigma = $\scriptsize $ \begin{pmatrix}
    \sqrt{\frac{-D}{m}} & -\sqrt{m} \\
    \sqrt{m} & \sqrt{\frac{-D}{m}}  \\
   \end{pmatrix}$\normalsize
   in Case A;
   \hfill
   $M_\sigma = $\scriptsize $  \begin{pmatrix}
    \frac{\sqrt{-D}+1}{\sqrt{2}} & \frac{D-1}{-4}\sqrt{2} \\
    \sqrt{2} & \frac{\sqrt{-D}-1}{\sqrt{2}}  \\
   \end{pmatrix}$\normalsize
   in Case B,
   \hfill
   ${}$
\end{center}
where we choose $m$ to be the smallest prime divisor of $D$.
The action on the boundary of hyperbolic space is given by classical M\"obius transformations,
\scriptsize
$ \begin{pmatrix}
    a & b \\ c & d
   \end{pmatrix} 
   \cdot z = \frac{az+b}{cz+d}.$ \normalsize \medspace
Therefore, Bianchi's additional symmetry flips the cusp $\infty$ with the cusp 
\begin{center}
\hfill
 $\sigma = \frac{\sqrt{-D}}{m}$ in Case A;
 \hfill
  $\sigma = \frac{\sqrt{-D}+1}{2}$ in Case B.
 \hfill
 ${}$.
\end{center}
We can check that the cusp $\sigma$ is singular by checking that the ideal $I_\sigma$ is not principal, where
\begin{center}
\hfill
 $I_\sigma = ({\sqrt{-D}},{m})$ in Case A;
 \hfill
  $I_\sigma = ({\sqrt{-D}+1}, {2})$ in Case B.
 \hfill
 ${}$
\end{center}
We assume the contrary and lead it to a contradiction:
Suppose that there are $c \in I_\sigma$, $a, b \in \ringO$ with 
\begin{center}
\hfill $a c = \sqrt{-D}$, $b c = m$  in Case A;
\hfill
$a c = \sqrt{-D}+1$, $b c = 2$ in Case B. \hfill ${}$ 
\end{center}
We easily show that $c$ is in the $\Z$-module with generators the two generators of $I_\sigma$
(we might say that ``$c$ admits coefficients in $\Z$'') when we use 
\\${}$ \hfill in Case A, that $m$ is a divisor of $D$; \hfill in Case B, that $D \equiv 1 \bmod 4$. \hfill ${}$ \\
Next, we separate real and imaginary part in the above system of two equations, solve for the coefficients of $c$ 
and arrive at the desired contradiction.
Therefore, the cusp $\sigma$ is singular. 
As we have chosen $m$ to be the smallest prime dividing $D$,
the cusp $\sigma$ is a vertex of the Bianchi fundamental polyhedron.
This entails that the matrix $M_\sigma$ 
of Bianchi's additional symmetry flips the Bianchi fundamental polyhedron onto itself,
interchanging two isometric halves of it which are separated by a plane equidistant to the cusps $\infty$ and~$\sigma$,
with respect to Mendoza's distance to cusps~\cite{Mendoza}.
The Bianchi group $\Gamma$ is a normal subgroup of index 2 in the group $\langle \Gamma, M_\sigma \rangle$,
whence this symmetry preserves the $\Gamma$-cell structure.
The other symmetries are obtained similarly~\cite{Bianchi1893}, and we collect them in Table~\ref{the table}.
The conditions in the table are implied by $\det(M_\sigma) = 1$.
Then we read off from the table that the radius $r_\sigma$ of the reflecting sphere is smaller than~$1$, and that $\sigma$ is a singular cusp (its numerator and denominator generate a non-principal ideal).

\end{proof}

\begin{table}
\scriptsize
$\begin{array}{|c|c|c|c|c|}
 \hline \text{Type} & \sigma & r_\sigma & M_\sigma & \text{Conditions} \\
 \hline \text{I}    & \frac{a_1 m +a_2\sqrt{-D}}{c m} & \frac{1}{c\sqrt{m}} & 
 \begin{pmatrix}
 a_1\sqrt{m}+a_2\sqrt{\frac{-D}{m}} & b\sqrt{{m}} \\
 c\sqrt{{m}} & a_2\sqrt{\frac{-D}{m}}-a_1\sqrt{m} \\
 \end{pmatrix} &
 \begin{array}{c} D \equiv 1\text{ or } 2 \bmod 4,\\ 
  \frac{D}{m} \text{ quadratic residue mod }m, \\
  m a_1^2+\frac{D}{m}a_2^2+m b c = 1\end{array}
 \\
\hline \text{II}    & \frac{a_2D -a_1 mc\sqrt{-D}}{cD} & \frac{1}{c\sqrt{\frac{D}{m}}} & 
 \begin{pmatrix}
 a_1\sqrt{m}+a_2\sqrt{\frac{-D}{m}} & b\sqrt{\frac{-D}{m}} \\
 c\sqrt{\frac{-D}{m}} & a_1\sqrt{m}-a_2\sqrt{\frac{-D}{m}} \\
 \end{pmatrix} &
 \begin{array}{c} D \equiv 1\text{ or } 2 \bmod 4,\\ 
  m \text{ quadratic residue mod }\frac{D}{m},\\
  m a_1^2 +\frac{D}{m} a_2^2 + \frac{D}{m}b c = 1 \end{array} \\
\hline \text{III}    & \frac{a_1 +a_2\sqrt{-D}}{2c} & \frac{1}{c\sqrt{2}} & 
 \begin{pmatrix}
 \frac{a_1+a_2\sqrt{-D}}{\sqrt{2}} & b\sqrt{2} \\
 c\sqrt{2} & \frac{a_2\sqrt{-D}-a_1}{\sqrt{2}} \\
 \end{pmatrix} &
 \begin{array}{c} D \equiv 1  \bmod 4,\\
 a_1^2+D a_2^2 +4bc = 2\\ 
\end{array} \\
\hline \text{IV}    & \frac{a_2 D -a_1\sqrt{-D} }{2cD} & \frac{1}{c\sqrt{2 D}} & 
 \begin{pmatrix}
 \frac{a_1+a_2\sqrt{-D}}{\sqrt{2}} & b\sqrt{2}\sqrt{-D} \\
 c\sqrt{2}\sqrt{-D} & \frac{a_1-a_2\sqrt{-D}}{\sqrt{2}} \\
 \end{pmatrix} &
 \begin{array}{c} D \equiv 1  \bmod 4\\ 
   2 \text{ quadratic residue mod } D,\\
   a_1^2 +Da_2^2 +4Dbc = 2
\end{array}  \\
\hline \text{V}    & \frac{a_1 m +a_2 \sqrt{-D}}{2cm} & \frac{1}{c\sqrt{2 m}} & 
 \begin{pmatrix}
 \frac{a_1\sqrt{m}+a_2\sqrt{\frac{-D}{m}}}{\sqrt{2}} & b\sqrt{2 m} \\
 c\sqrt{2 m} & \frac{-a_1\sqrt{m} +a_2\sqrt{\frac{-D}{m}}}{\sqrt{2}} \\
 \end{pmatrix} &
 \begin{array}{c} D \equiv 1  \bmod 4\\ 
   2\text{ and }\frac{D}{m} \text{ have the same}\\
   \text{quadratic character mod } m,\\
   m a_1^2 +\frac{D}{m}a_2^2 +4mbc = 2
\end{array}  \\
\hline \text{VI}    & \frac{a_2D +a_1 m\sqrt{-D}}{2cD} & \frac{1}{c\sqrt{2 \frac{D}{m}}} & 
 \begin{pmatrix}
 \frac{a_1\sqrt{m}+a_2\sqrt{\frac{-D}{m}}}{\sqrt{2}} & b\sqrt{2 \frac{-D}{m}} \\
 c\sqrt{2 \frac{-D}{m}} & \frac{a_1\sqrt{m} -a_2\sqrt{\frac{-D}{m}}}{\sqrt{2}} \\
 \end{pmatrix} &
 \begin{array}{c} D \equiv 1  \bmod 4\\ 
   2\text{ and }m \text{ have the same}\\
   \text{quadratic character mod } \frac{D}{m},\\
   m a_1^2 +\frac{D}{m}a_2^2 +4\frac{D}{m}bc = 2
\end{array}  \\
\hline \text{VII}    & \frac{a_1 m +a_2 \sqrt{-D}}{2cm} & \frac{1}{c\sqrt{m}} & 
 \begin{pmatrix}
 \frac{a_1\sqrt{m}+a_2\sqrt{\frac{-D}{m}}}{2} & b\sqrt{m} \\
 c\sqrt{m} & \frac{-a_1\sqrt{m} +a_2\sqrt{\frac{-D}{m}}}{2} \\
 \end{pmatrix} &
 \begin{array}{c} D \equiv 3  \bmod 4\\ 
  \frac{D}{m} \text{ quadratic residue mod }m,\\
   m a_1^2 +\frac{D}{m}a_2^2 +4mbc = 4
\end{array}  \\
\hline \text{VIII}    & \frac{a_2 D +a_1 m\sqrt{-D}}{2cD} & \frac{1}{c\sqrt{\frac{D}{m}}} & 
 \begin{pmatrix}
 \frac{a_1\sqrt{m}+a_2\sqrt{\frac{-D}{m}}}{2} & b\sqrt{\frac{-D}{m}} \\
 c\sqrt{\frac{-D}{m}} & \frac{a_1\sqrt{m} -a_2\sqrt{\frac{-D}{m}}}{2} \\
 \end{pmatrix} &
 \begin{array}{c} D \equiv 3  \bmod 4\\ 
  m \text{ quadratic residue mod }\frac{D}{m},\\
   m a_1^2 +\frac{D}{m}a_2^2 +4\frac{D}{m}bc = 4
\end{array}  \\
\hline
\end{array}$\normalsize
\caption{Bianchi's additional symmetries~\cite{Bianchi1893}*{\S III}. The parameters $a_1, a_2, b, c \in \Z$ and the divisor $m$ of $D$ have to satisfy the specified conditions.
Examples for Type I are \hfill $m = 2, a_1 = 0, a_2 = c = 1$, $(D,b) \in \{(6,-1), (10,-2), (22,-5), (58,-14)\}$;
examples for Type III are $a_1 = a_2 = c = 1$, $(D,b) \in \{(5,-1), (13,-3), (37,-9)\}$.} \label{the table}
\end{table}

\section{Proof of Lemma 2}\label{The proof}
We will use Bianchi's additional symmetries for decomposing the Bianchi fundamental polyhedron into two isometric parts
with the cutting plane between them being equidistant between $\infty$ and the singular cusp at which $\T_\sigma$ is located.

\begin{proof}[Proof of lemma \ref{torus-lemma}]
 Consider, in the boundary of the Bianchi fundamental polyhedron, the fundamental rectangle $\mathcal F$ for the action of the cusp stabiliser on the plane joined to $\Hy$ at the cusp~$\infty$.
 There is a sequence of rectangles in $\Hy$ obtained as translates of~$\mathcal F$ orthogonal to all the geodesic arcs emanating from the cusp~$\infty$.
 This way, the portion of the fundamental polyhedron which touches the cusp~$\infty$, is locally homeomorphic to the Cartesian product of a geodesic arc with a translate of~$\mathcal F$.

 The boundaries of these translates are subject to the same identifications by $\Gamma$ as the boundary of~$\mathcal F$.
 Namely, denote by $t({\mathcal F})$ one of the translates of~$\mathcal F$; and denote by 
 $\gamma_x$ and $\gamma_y$ two generators, up to the $-1$ matrix, of the cusp stabiliser identifying the opposite edges of~$\mathcal F$.
We can choose $\gamma_x =$\scriptsize$\begin{pmatrix} 1 & 1 \\ 0 & 1 \end{pmatrix}$ \normalsize 
and $\gamma_y =$\scriptsize$\begin{pmatrix} 1 & \omega \\ 0 & 1 \end{pmatrix}$, \normalsize
with $\omega = \sqrt{-D}$ for $D \equiv 1\text{ or }2 \bmod 4$, respectively $\omega = \frac{\sqrt{-D}+1}{2}$ for $D \equiv 3 \bmod 4$.
Then,  $\gamma_x$ and $\gamma_y$ identify the opposite edges of $t({\mathcal F})$
 and make the quotient of $t({\mathcal F})$ into a torus.
  So in the quotient space by the action of  $\Gamma$, the image of~$\T_\infty$ is wrapped into a sequence of layers of tori.
 And therefore in turn, the 3-dimensional interior of the  Bianchi fundamental polyhedron is wrapped around the image of~$\T_\infty$
 along the entire surface of the latter, such that there is a neighbourhood inside $_\Gamma \backslash \Hy$ which is homeomorphic to the Cartesian product of a $2$-torus and an open interval.
 Hence, there is, in a suitable ambient space, a neighbourhood of the image of~$\T_\infty$ that is homeomorphic to Euclidean $3$--space with the interior of a solid torus removed.
 Now, considering the cell structure of the torus, we see that precisely one of the loops generating the fundamental group of~$\T_\infty$
 can be contracted in the interior of the quotient of the Bianchi fundamental polyhedron --- namely, the loop given by the edge that has its endpoints identified by $\gamma_x$.
 The other loop (the one obtained by the identification by  $\gamma_y$) is nontrivially linked with the removed solid torus.
 This entails that it cannot be unlinked when moving it in the Borel--Serre compactification of~$_\Gamma \backslash \Hy$,
 and thus remains uncontractible there.
 For each ideal class that is represented by a singular cusp $\sigma$ which is subject to one of Bianchi's additional symmetries, we observe the following.
  The radius $r_\sigma$ of the reflecting sphere is smaller than~$1$ (cf. Table~\ref{the table}), 
  and hence the above described contraction happens \textit{within} the half touching the cusp~$\infty$.
 By our $\Gamma$-cell structure preserving isometry, we get a copy of this contraction.
 The r\^oles of~$\gamma_x$ and~$\gamma_y$ are then played by~$M_\sigma\gamma_xM_\sigma^{-1}$ and~$M_\sigma\gamma_y M_\sigma^{-1}$ with $M_\sigma$ specified in Table~\ref{the table}.
 Hence the inclusion of~$\T_\sigma$ into $_\Gamma \backslash \widehat{\Hy}$ makes exactly one of the edges of~$\T_\sigma$ become the boundary of a $2$--chain.
 The number of such cusps~$\sigma$ is given by Theorem~\ref{Bianchi statement}.
 And Serre's theorem about the rank of the map $\alpha$ \cite{Serre}*{th\'eor\`eme 7} allows us to conclude that no nontrivial linear combination of these loops can be homologous to zero.
 \end{proof}

Note that only with Bianchi's additional symmetries, we get more than one linked loop.
Sufficiently many of these symmetries are available only for the Bianchi orbifolds specified in Corollary~\ref{exponent 2 case},
and in general we just have the one linked loop at the torus at infinity, which was already observed in Serre's original paper.

\subsection*{Acknowledgements}
The author would like to thank Lizhen Ji for a fruitful correspondence which laid the foundations for the present paper.
He would like to heartily thank the anonymous referee for the many very knowledgeable and helpful suggestions made on various aspects.
He is grateful to the University of Luxembourg for funding his research through Gabor Wiese's AMFOR grant.

\end{document}